\documentclass[english]{article}
\usepackage[T1]{fontenc}
\usepackage[latin9]{inputenc}
\usepackage{babel}

\usepackage{url}
\usepackage{amsthm}
\usepackage{amsmath}
\usepackage{graphicx}
\usepackage{amssymb}
\usepackage[all]{xy}
\usepackage[unicode=true, pdfusetitle,
 bookmarks=true,bookmarksnumbered=false,bookmarksopen=false,
 breaklinks=false,pdfborder={0 0 1},backref=false,colorlinks=false]
 {hyperref}

\makeatletter
  \theoremstyle{plain}
  \newtheorem*{thm*}{Theorem}
  \theoremstyle{remark}
  \newtheorem*{acknowledgement*}{Acknowledgement}
\theoremstyle{plain}
\newtheorem{thm}{Theorem}[section]
  \theoremstyle{plain}
  \newtheorem{prop}[thm]{Proposition}
  \theoremstyle{plain}
  \newtheorem{lem}[thm]{Lemma}
  \theoremstyle{plain}
  \newtheorem{cor}[thm]{Corollary}
  \theoremstyle{remark}
  \newtheorem{rem}[thm]{Remark}

\date{}

\makeatother

\begin{document}

\title{The contact homology of Legendrian knots with maximal Thurston-Bennequin
invariant}

\author{Steven Sivek}
\maketitle
\begin{abstract}
We show that there exists a Legendrian knot with maximal Thurston-Bennequin
invariant whose contact homology is trivial. We also provide another
Legendrian knot which has the same knot type and classical invariants
but nonvanishing contact homology.
\end{abstract}

\section{Introduction}

The Chekanov-Eliashberg invariant \cite{Chekanov,Eliashberg}, which
assigns to each Legendrian knot $K$ a differential graded algebra
$(Ch(K),\partial)$ over $\mathbb{F}=\mathbb{Z}/2\mathbb{Z}$, has
been a powerful tool for classifying Legendrian knots in the standard
contact $S^{3}$. The closely related characteristic algebra $\mathcal{C}(K)$
was defined by Ng \cite{Ng - computable} to be the quotient of $Ch(K)$
by the two-sided ideal $\langle\mathrm{Im}(\partial)\rangle$; if
two knots $K$ and $K'$ are Legendrian isotopic, then we can add
some free generators to $\mathcal{C}(K)$ and $\mathcal{C}(K')$ to
make them tamely isomorphic. Both of these invariants only provide
information about nondestabilizable knots: if $K$ is a stabilized
knot, then both the Legendrian contact homology $H_{*}(Ch(K))$ and
the characteristic algebra $\mathcal{C}(K)$ vanish. For an introduction
to Legendrian knots, see \cite{Etnyre-intro}.

Shonkwiler and Vela-Vick \cite{SV} gave the first examples of Legendrian
knots with nonvanishing contact homology which do not have maximal
Thurston-Bennequin invariant, representing the knot types $m(10_{161})$
and $m(10_{145})$. Conversely, there are conjecturally nondestabilizable
knots of type $m(10_{139})$, $10_{161}$, and $m(12n_{242})$ with
non-maximal $tb$ and vanishing contact homology \cite{atlas,SV}.
On the other hand, it is an open question whether there is a Legendrian
knot $K$ for which $tb(K)$ is maximal but the contact homology of
$K$ vanishes. We will answer this question and show that it is not
determined solely by the classical invariants $tb$ and $r$ of $K$:
\begin{thm*}
There are distinct $tb$-maximizing Legendrian representatives $K_{1}$
and $K_{2}$ of $m(10_{132})$ with the same classical invariants
such that $K_{1}$ has trivial contact homology, even with $\mathbb{Z}[t,t^{-1}]$
coefficients, while $K_{2}$ does not.
\end{thm*}
These Legendrian knots, found in Chongchitmate and Ng's atlas of Legendrian
knots \cite{atlas}, can be specified as plat diagrams by the following
braid words: \begin{eqnarray*}
K_{1}: &  & 6,7,4,3,7,5,3,6,4,2,5,1,3,2,5,2,4,6,2\\
K_{2}: &  & 4,5,3,5,3,2,4,1,3,2,4,2,5,1,3,2,4,4,3,5,4,2\end{eqnarray*}
Indeed, both knots have classical invariants $tb=-1$ and $r=0$,
and Ng \cite{Ng - max tb} showed that $\overline{tb}(m(10_{132}))=-1$
by bounding $\overline{tb}$ for an appropriate cable of $m(10_{132})$.
We will prove this theorem in Section \ref{sec:The-m10_132-examples}.

Finally, the proof that $K_{2}$ has nonvanishing contact homology
uses an action of $\mathcal{C}(K_{2})$ on an infinite-dimensional
vector space, just as the nonvanishing examples in \cite{SV} did.
In Section \ref{sec:fd-reps} we will show that this is necessary
in the sense that $\mathcal{C}(K_{2})$ does not have any finite-dimensional
representations. It is completely understood when a characteristic
algebra $\mathcal{C}$ does not have any $1$-dimensional representations,
but we will ask if such a $\mathcal{C}$ can admit maps $\mathcal{C}\to\mathrm{Mat}_{n}(\mathbb{F})$
for some finite $n\geq2$. We will show that this is possible in general
by constructing $2$-dimensional representations for specific Legendrian
representatives of negative torus knots.
\begin{acknowledgement*}
I would like to thank Ana Caraiani, Tom Mrowka, Lenny Ng, Clayton
Shonkwiler, and David Shea Vela-Vick for helpful comments. This work
was supported by an NSF Graduate Research Fellowship.
\end{acknowledgement*}

\section{The $m(10_{132})$ examples\label{sec:The-m10_132-examples}}

\subsection{The vanishing example\label{sub:10_132-1}}

Let $K_{1}$ be the Legendrian representative of $m(10_{132})$ in
Figure \ref{fig:10_132-1}. Its Chekanov-Eliashberg algebra is generated
freely over $\mathbb{Z}[t,t^{-1}]$ by elements $x_{1},\dots,x_{23}$
with differentials specified in Appendix \ref{sec:differential-10_132-1}.

\noindent \begin{center}
\begin{figure}[h]
\begin{centering}
\includegraphics[scale=0.75]{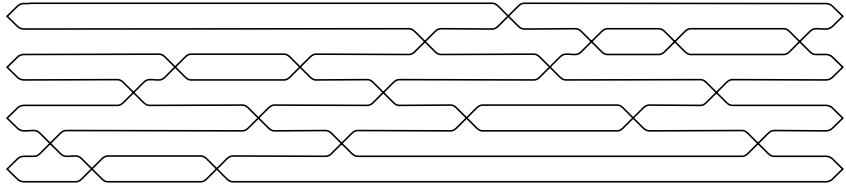}
\par\end{centering}

\caption{\label{fig:10_132-1}The representative $K_{1}$ of $m(10_{132})$,
defined by the braid word $6,7,4,3,7,5,3,6,4,2,5,1,3,2,5,2,4,6,2$.}

\end{figure}

\par\end{center}

To show that $K_{1}$ has vanishing contact homology, we need to find
a relation $\partial x=1$ in $Ch(K_{1})$. Recall that $Ch(K_{1})$
uses a signed Leibniz rule $\partial(vw)=(\partial v)w+(-1)^{|v|}v(\partial w)$,
where $|v|$ is the grading of the homogeneous element $v$, and note
that the generators with odd grading are \[
x_{2},x_{3},x_{5},x_{9},x_{11},x_{12},x_{13},x_{15},x_{20},x_{21},x_{22},x_{23}.\]
Let $a=x_{12}(x_{4}(1+x_{2}x_{5})-x_{8})+x_{14}x_{5}$ and $b=\partial a$;
then $b=x_{10}x_{4}(1+x_{2}x_{5})-x_{10}x_{8}+x_{13}x_{5}$ and $\partial b=0$.
Now \begin{eqnarray*}
\partial(x_{22}+x_{12}-ax_{18}) & = & 1+x_{17}x_{7}+bx_{18}-(bx_{18}+ax_{15}x_{7})\\
 & = & 1+(x_{17}-ax_{15})x_{7};\end{eqnarray*}
let $c=x_{22}+x_{12}-ax_{18}$. Since \[
\partial(x_{17}-ax_{15})=\partial(x_{6}-x_{4}x_{1})=\partial(1+x_{16}x_{19})=0,\]
and $\partial x_{20}=1+(x_{6}-x_{4}x_{1})(1+x_{16}x_{19})$, we can
compute \[
\partial\left(x_{20}-c(x_{6}-x_{4}x_{1})(1+x_{16}x_{19})\right)=1-(x_{17}-ax_{15})x_{7}(x_{6}-x_{4}x_{1})(1+x_{16}x_{19}).\]
Finally, we have $-x_{7}(x_{6}-x_{4}x_{1})=\partial(x_{9}+x_{2})$,
so we conclude that \[
\partial\left(x_{20}-\left(c(x_{6}-x_{4}x_{1})+(x_{17}-ax_{15})(x_{9}+x_{2})\right)(1+x_{16}x_{19})\right)=1,\]
and so $K_{1}$ has trivial contact homology over $\mathbb{Z}[t,t^{-1}]$
as desired.

\subsection{The nonvanishing example\label{sub:10_132-2}}

Let $K_{2}$ be the Legendrian representative of $m(10_{132})$ in
Figure \ref{fig:10_132-2}. The algebra $Ch(K_{2})$ is generated
freely over $\mathbb{F}=\mathbb{Z}/2\mathbb{Z}$ by $x_{1},\dots,x_{25}$
with differentials specified in Appendix \ref{sec:differential-10_132-2}.
In order to show that $K_{2}$ has nontrivial contact homology, it
will suffice to show that the characteristic algebra $\mathcal{C}_{2}=\mathcal{C}(K_{2})$
is nonvanishing \cite{SV}.

\noindent %
\begin{figure}[h]
\begin{centering}
\includegraphics[scale=0.7]{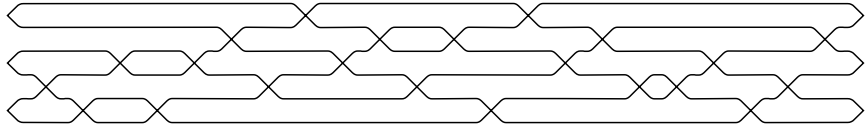}
\par\end{centering}

\caption{\label{fig:10_132-2}The representative $K_{2}$ of $m(10_{132})$,
defined by the braid word $4,5,3,5,3,2,4,1,3,2,4,2,5,1,3,2,4,4,3,5,4,2$.}

\end{figure}

The differential in $\mathcal{C}_{2}$ immediately gives us $x_{1}=x_{6}=0$,
and \[
x_{12}=\partial(x_{12}x_{23}+x_{15}x_{22}+x_{17}x_{18})\]
gives $x_{12}=0$, hence $\partial x_{24}=0$ becomes $(1+x_{5}(x_{2}+x_{3}))x_{20}=1$.
Then we can use $(\partial x_{13})x_{20}=0$ and $(\partial x_{17})x_{20}=0$
to get $x_{11}=0$ and $x_{15}=0$, so \[
x_{1}=x_{6}=x_{11}=x_{12}=x_{15}=0.\]
Furthermore, $\partial x_{21}=0$ becomes $x_{14}=cx_{20}$, so $\partial x_{25}=0$
gives us $x_{14}=x_{20}$.

Consider the quotient of $\mathcal{C}_{2}$ by the two-sided ideal
\[
\mathcal{I}=\langle x_{3},x_{7},x_{8},x_{9},x_{10},x_{13}+1+x_{2}x_{5},x_{17},x_{19},x_{21},\dots,x_{25}\rangle.\]
The quotient $\mathcal{C}_{2}/\mathcal{I}$ is generated by $x_{2},x_{4},x_{5},x_{14},x_{16},x_{18}$,
and its nontrivial relations are $c=x_{2}+x_{14}(1+x_{2}x_{5})+x_{16}(1+x_{5}x_{2})=1$
and \begin{eqnarray*}
x_{4} & = & x_{5}(1+x_{2}x_{4})\\
x_{18} & = & 1+x_{2}x_{4}\\
0 & = & (1+x_{5}x_{2})x_{18}\\
1 & = & (1+x_{2}x_{5})x_{18}\\
1 & = & (1+x_{5}x_{2})x_{14}.\end{eqnarray*}
Note that the pair of relations $x_{4}=x_{5}(1+x_{2}x_{4})$ and $x_{18}=1+x_{2}x_{4}$
are equivalent to $x_{4}=x_{5}x_{18}$ and $(1+x_{2}x_{5})x_{18}=1$,
the latter of which is already known, so we can replace the pair with
$x_{4}=x_{5}x_{18}$. Furthermore, multiplying the $c=1$ equation
on the right by $x_{18}$ gives $x_{14}=(1+x_{2})x_{18}$, hence the
last relation becomes $(1+x_{5}x_{2})x_{2}x_{18}=1$. Then the $c=1$
equation becomes \[
x_{16}(1+x_{5}x_{2})=(1+x_{2})(1+x_{18}(1+x_{2}x_{5}))\]
so we multiply on the right by $x_{2}x_{18}$ and get \[
x_{16}=(1+x_{2})(x_{2}x_{18}+x_{18}x_{2}(1+x_{5}x_{2})x_{18})=(1+x_{2})x_{2}x_{18}.\]
Thus we see that $x_{4}$, $x_{14}$, and $x_{16}$ can be expressed
in terms of $x_{2}$, $x_{5}$, and $x_{18}$, and $c=1$ can be rewritten
as \[
0=(1+x_{2})\left(1+x_{18}(1+x_{2}x_{5})+x_{2}x_{18}(1+x_{5}x_{2})\right).\]
Relabeling $x_{2},x_{5},x_{18}$ as $a,b,c$ respectively, we have
a homomorphism from $\mathcal{C}_{2}/\mathcal{I}$ to the quotient
$R$ of the free algebra $\mathbb{F}\langle a,b,c\rangle$ by the
two-sided ideal generated by the relations \begin{eqnarray*}
0 & = & 1+c(1+ab)+ac(1+ba)\\
0 & = & (1+ba)c\\
1 & = & (1+ab)c\\
1 & = & (1+ba)ac.\end{eqnarray*}

\begin{prop}
The algebra $R$ is nontrivial.\end{prop}
\begin{proof}
We will construct an infinite-dimensional representation of $R$,
following ideas from \cite{SV}. Let $\mathcal{H}$ be a countable-dimensional
$\mathbb{F}$-vector space, with basis $\{v_{0},v_{1},v_{2},\dots\}$,
and write $\mathcal{H}=\mathcal{H}_{1}\oplus\mathcal{H}_{2}$ where
each $\mathcal{H}_{i}$ summand is isomorphic to $\mathcal{H}$. Let
$f,g:\mathcal{H}\to\mathcal{H}$ be homomorphisms defined by $f(v_{i})=v_{2i}$
and $g(v_{i})=v_{2i+1}$, so that the diagrams \begin{eqnarray*}
\xymatrix{\mathcal{H}_{1}\ar@{}[d]|{\bigoplus}\ar[r]^{f} & \mathcal{H}_{1}\ar@{}[d]|{\bigoplus}\\
\mathcal{H}_{2}\ar[ru]^{g} & \mathcal{H}_{2}}
 &  & \xymatrix{\mathcal{H}_{1}\ar@{}[d]|{\bigoplus}\ar[rd]^{f} & \mathcal{H}_{1}\ar@{}[d]|{\bigoplus}\\
\mathcal{H}_{2}\ar[r]^{g} & \mathcal{H}_{2}}
\end{eqnarray*}
represent isomorphisms $\mathcal{H}\stackrel{\sim}{\to}\mathcal{H}_{1}$
and $\mathcal{H}\stackrel{\sim}{\to}\mathcal{H}_{2}$, respectively.
We also define homomorphisms $p,s:\mathcal{H}\to\mathcal{H}$ by $p(v_{i})=v_{i-1}$
for $i\geq1$, $p(v_{0})=0$ and $s(v_{i})=v_{i+1}+v_{2(i+1)}$. It
is straightforward to check the identities \begin{eqnarray*}
s\circ p=f+1, & p\circ g=f, & p\circ s=g+1.\end{eqnarray*}

We define a right action of $a$ and $b$ on $\mathcal{H}\cong\mathcal{H}_{1}\oplus\mathcal{H}_{2}$
by the diagrams \begin{eqnarray*}
\xymatrix{\mathcal{H}_{1}\ar@{}[d]|{\bigoplus}\ar[rd]|(0.4){p} & \mathcal{H}_{1}\ar@{}[d]|{\bigoplus}\\
\mathcal{H}_{2}\ar[ru]|(0.4){1} & \mathcal{H}_{2}}
 & \xymatrix{\ \ar@{}[d]|{\displaystyle \textrm{and}}\\
\ }
 & \xymatrix{\mathcal{H}_{1}\ar@{}[d]|{\bigoplus}\ar[r]^{g}\ar[rd]|(0.4){1} & \mathcal{H}_{1}\ar@{}[d]|{\bigoplus}\\
\mathcal{H}_{2}\ar[ru]|(0.4){s} & \mathcal{H}_{2}}
\end{eqnarray*}
respectively. Then we can compute the action of $ab$ and $ba$ by
concatenating the $a$ and $b$ diagrams to get \begin{eqnarray*}
\xymatrix{\mathcal{H}_{1}\ar@{}[d]|{\bigoplus}\ar[r]^{s\circ p} & \mathcal{H}_{1}\ar@{}[d]|{\bigoplus}\\
\mathcal{H}_{2}\ar[ru]^{g}\ar[r]^{1} & \mathcal{H}_{2}}
 & \xymatrix{\ \ar@{}[d]|{{\displaystyle \textrm{and}}}\\
\ }
 & \xymatrix{\mathcal{H}_{1}\ar@{}[d]|{\bigoplus}\ar[r]^{1}\ar[rd]^{p\circ g} & \mathcal{H}_{1}\ar@{}[d]|{\bigoplus}\\
\mathcal{H}_{2}\ar[r]^{p\circ s} & \mathcal{H}_{2}}
\end{eqnarray*}
respectively, hence by the above identities $1+ab$ and $1+ba$ are
exactly the specified isomorphisms $\mathcal{H}\stackrel{\sim}{\to}\mathcal{H}_{1}$
and $\mathcal{H}\stackrel{\sim}{\to}\mathcal{H}_{2}$. Finally, let
$c$ act on $\mathcal{H}$ as the map \[
\xymatrix{\mathcal{H}_{1}\ar@{}[d]|{\bigoplus}\ar!<30pt,-15pt>*{}^{\sim} & \ \ar@{}[d]|{{\displaystyle \mathcal{H}}}\\
\mathcal{H}_{2}\ar!<30pt,15pt>*{}_{0} & \ }
\]
where the indicated isomorphism is the inverse of $\mathcal{H}\stackrel{\sim}{\to}\mathcal{H}_{1}$.
Then the composition $ac$ is the homomorphism \[
\xymatrix{\mathcal{H}_{1}\ar@{}[d]|{\bigoplus}\ar!<30pt,-15pt>*{}^{0} & \ \ar@{}[d]|{\displaystyle \mathcal{H}}\\
\mathcal{H}_{2}\ar!<30pt,15pt>*{}_{\sim} & \ }
\]
where the isomorphism is inverse to $\mathcal{H}\stackrel{\sim}{\to}\mathcal{H}_{2}$.
It is now easy to check that $(1+ab)c=1$, $(1+ba)c=0$, and $(1+ba)ac=1$.
Finally, we note that $c(1+ab)$ is the projection of $\mathcal{H}$
onto $\mathcal{H}_{1}\subset\mathcal{H}$ and likewise $ac(1+ba)$
is the projection onto $\mathcal{H}_{2}$, hence \[
1=c(1+ab)+ac(1+ba).\]
Therefore the action which we have constructed satisfies all of the
defining relations of $R$.
\end{proof}
Since $R$ is nonvanishing and we have a homomorphism $\mathcal{C}_{2}\to\mathcal{C}_{2}/\mathcal{I}\to R$,
we conclude that $\mathcal{C}_{2}$ (and hence the contact homology
of $K_{2}$) is nonvanishing as well.

\section{\label{sec:fd-reps}Finite-dimensional representations of $\mathcal{C}(K)$}

Although the Legendrian knot $K_{2}$ of Section \ref{sub:10_132-2}
is now known to have nontrivial contact homology and characteristic
algebra, one can ask for a simpler proof of this fact; in particular,
one can ask if $\mathcal{C}_{2}$ has any finite-dimensional representations.
The answer in this case is no.
\begin{lem}
\label{lem:no-rep-criterion}Suppose that an $\mathbb{F}$-algebra
$\mathcal{A}$ has a relation of the form $ab=1$. If the quotient
of $\mathcal{A}$ by the two-sided ideal $\langle ba-1\rangle$ is
trivial, i.e. if $0=1$ in $\mathcal{A}/\langle ba-1\rangle$, then
there is no representation $\mathcal{A}\to\mathrm{Mat}_{n}(\mathbb{F})$
for any $n$.\end{lem}
\begin{proof}
Suppose there is a homomorphism $\varphi:\mathcal{A}\to\mathrm{Mat}_{n}(\mathbb{F})$,
so in particular $\varphi(1)=1$. The equation $\varphi(ab-1)=0$
implies that $\varphi(a)$ and $\varphi(b)$ are inverse matrices,
so they commute and $\varphi(ba-1)=0$ as well. Then $\varphi$ factors
through the quotient $\mathcal{A}/\langle ba-1\rangle$ in which $0=1$,
hence $\varphi(1)=\varphi(0)=0$, which is a contradiction.
\end{proof}
Now in $\mathcal{C}_{2}$, we showed in Section \ref{sub:10_132-2}
that $x_{11}=x_{12}=0$ and $(1+x_{5}(x_{2}+x_{3}))x_{20}=1$. If
we impose the relation $x_{20}(1+x_{5}(x_{2}+x_{3}))=1$, then $x_{18}=x_{20}(\partial x_{22})=0$
as well and so $0=\partial x_{23}=1$, hence $\mathcal{C}_{2}$ has
no finite-dimensional representations by Lemma \ref{lem:no-rep-criterion}.

Lemma \ref{lem:no-rep-criterion} can also be used to prove that the
characteristic algebra of the Legendrian $m(10_{161})$ studied in
\cite{SV} has no finite-dimensional representations, by adding $x_{28}x_{13}=1$
to the relations $\partial x_{i}=0$ in \cite[Appendix A]{SV} and
showing that $0=1$ as a consequence, and similarly for the $m(10_{145})$
representative mentioned in the same article. Neither one of these
knots has maximal Thurston-Bennequin invariant.

On the other hand, it is interesting to ask when the characteristic
algebra $\mathcal{C}$ of a Legendrian knot $K$ has $n$-dimensional
representations. For $n=1$ the answer depends only on $tb$ and the
topological knot type:
\begin{prop}
\label{pro:kauffman-bound}There is a homomorphism $\mathcal{C}\to\mathrm{Mat}_{1}(\mathbb{F})\cong\mathbb{F}$
if and only if the Kauffman bound \[
tb(K)\leq\textrm{\emph{min-deg}}_{a}F_{K}(a,x)-1\]
(see \cite{FT-Kauffman}) is sharp.\end{prop}
\begin{proof}
The Kauffman bound for $K$ is achieved if and only if a front diagram
for $K$ admits an ungraded normal ruling \cite{Rutherford}, which
happens if and only if $Ch(K)$ admits an ungraded augmentation \cite{Fuchs,Fuchs-Ishkhanov,Sabloff}.
An augmentation is an algebra homomorphism $Ch(K)\stackrel{\epsilon}{\to}\mathbb{F}$
which satisfies $\epsilon\circ\partial=0$, and these correspond bijectively
to algebra homomorphisms $\mathcal{C}\to\mathbb{F}$, so the latter
exists if and only if the Kauffman bound is sharp.
\end{proof}
In particular, the Kauffman bound is known to be sharp for all knots
with at most 9 crossings except for $m(8_{19})$ and $m(9_{42})$
(see \cite{Ng - 2-bridge}); for all 10-crossing knots except $m(10_{124})$,
$m(10_{128})$, $m(10_{132})$, and $m(10_{136})$ \cite{KnotInfo};
and for all alternating knots \cite{Rutherford}. Thus the characteristic
algebra of a Legendrian representative of one of these knot types
has a 1-dimensional representation if and only if it is $tb$-maximizing.

We will now demonstrate the existence of infinitely many Legendrian
knots whose characteristic algebras have $n$-dimensional representations
for $n=2$ but not for $n=1$. For convenience, we will use the following
presentation of $\mathrm{Mat}_{2}(\mathbb{F})$.
\begin{lem}
The ring $\mathrm{Mat}_{2}(\mathbb{F})$ has a presentation of the
form \[
\frac{\mathbb{F}\langle a,b\rangle}{\langle a^{2}=b^{2}=0,ab+ba=1\rangle}.\]
\end{lem}
\begin{proof}
Let $R$ be the $\mathbb{F}$-algebra with the given presentation,
and consider a map $\varphi:R\to\mathrm{Mat}_{2}(\mathbb{F})$ of
the form \begin{eqnarray*}
a & \mapsto & A=\left(\begin{array}{cc}
0 & 1\\
0 & 0\end{array}\right)\\
b & \mapsto & B=\left(\begin{array}{cc}
0 & 0\\
1 & 0\end{array}\right).\end{eqnarray*}
It is easy to check that $A^{2}=B^{2}=0$ and $AB+BA=I$, so $\varphi$
is a valid homomorphism, and since $A,B,AB,BA$ form an additive basis
of $\mathrm{Mat}_{2}(\mathbb{F})$ it is surjective. To check that
$\varphi$ is also injective, we note that any nonzero monomial in
$R$ is equal to one of $1,a,b,ab,$ or $ba=1+ab$, and so $1,a,b,ab$
span $R$ as an $\mathbb{F}$-vector space; since the image of $\varphi$
has order $|\mathrm{Mat}_{2}(\mathbb{F})|=16\geq|R|$ it follows that
$\varphi$ is injective.
\end{proof}
Let $T_{p,-q}$ be the Legendrian representative of the $(p,-q)$-torus
knot as in Figure \ref{fig:torus-knot}, where $q>p\geq3$; there
are $p$ numbered left cusps at the leftmost edge of the diagram,
$q-p$ left cusps in the innermost region of the diagram, and $q$
right cusps. The algebra $Ch(T_{p,-q})$ can be computed following
\cite{Ng - computable}: the front projection is simple, so $Ch(T_{p,-q})$
is generated by crossings and right cusps and the differential counts
admissible embedded disks in the diagram.

\begin{figure}
\begin{centering}
\includegraphics{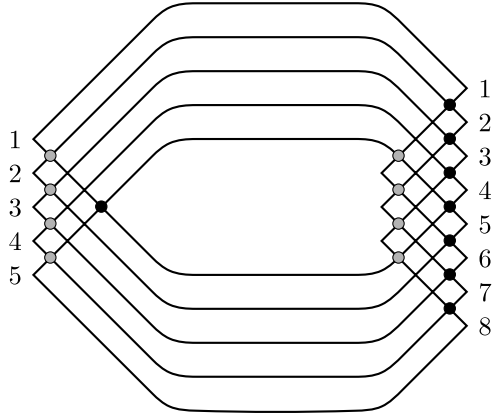}
\par\end{centering}

\caption{A Legendrian representative $T_{5,-8}$ of the $(5,-8)$-torus knot.\label{fig:torus-knot}}

\end{figure}
We label the generators of $Ch(T_{p,-q})$ as follows. On the left
half of the diagram, $x_{ij}$ is the intersection of the strands
through the numbered left cusps $i$ and $j$ for $1\leq i<j\leq p$.
On the right half, $y_{ij}$ denotes the intersection of strands through
the numbered right cusps $i$ and $j$ for $1\leq i<j\leq\max(q,i+p-1)$,
and $z_{i}$ is the $i$th right cusp. 

We define an algebra homomorphism $f:Ch(T_{p,-q})\to\mathrm{Mat}_{2}(\mathbb{F})$
by sending all generators to $0$ except \begin{eqnarray*}
x_{i,i+1},y_{j,j+p-1} & \mapsto & a\\
x_{1,p},y_{j,j+1} & \mapsto & b.\end{eqnarray*}
In Figure \ref{fig:torus-knot}, $f$ is equal to $a$ on the crossings
marked with gray dots, $b$ on the crossings marked with black dots,
and $0$ on all other crossings and right cusps. If we can show that
$f(\partial v)=0$ for all generators $v$, then $f$ is a morphism
of DGAs (where $\mathrm{Mat}_{2}(\mathbb{F})$ has trivial differential)
and it induces a representation $\mathcal{C}(T_{p,-q})\to\mathrm{Mat}_{2}(\mathbb{F})$.
\begin{prop}
The homomorphism $f:Ch(T_{p,-q})\to\mathrm{Mat}_{2}(\mathbb{F})$
satisfies $f(\partial v)=0$ for all $v$.\end{prop}
\begin{proof}
Call an admissible disk \emph{nontrivial} if none of its corners are
in $\ker(f)$. Then it is easy to see that any nontrivial disk has
exactly two corners, and if both corners have the same color (in the
sense of Figure \ref{fig:torus-knot}, i.e. if they are sent to the
same element of $\mathrm{Mat}_{2}(\mathbb{F})$) then the contribution
of this disk to $f(\partial v)$ is either $a^{2}=0$ or $b^{2}=0$.
Thus we can determine $f(\partial v)$ by only counting disks with
initial vertex at $v$ and having exactly one gray corner and one
black corner.

If $v$ is the right cusp $z_{i}$, then there are two nontrivial
disks contributing $ab$ and $ba$ to the differential, so $f(\partial z_{i})=1+ab+ba=0$.
For all crossings $v$, however, the only possible black corner for
a nontrivial disk is $x_{1,p}$. Such a disk must include either the
first or the $p$th numbered left cusp on its boundary depending on
whether the interior of the disk is immediately above or below $x_{1,p}$,
but then the boundary of the disk must pass through either $z_{1}$
or $z_{q}$, which in particular is to the right of $v$, and so it
cannot contribute to $f(\partial v)$. We conclude that $f(\partial v)=0$
for all generators $v$ of $Ch(T_{p,-q})$, as desired.
\end{proof}
We can compute $tb(T_{p,-q})=-pq$ for all $p$ and $q$, hence $T_{p,-q}$
is $tb$-maximizing by the classification of Legendrian torus knots
\cite{Torus knots}, but for odd $p$ the Kauffman bound is $tb(K)\leq-pq+q-p$
\cite{Epstein-Fuchs}. Using Proposition \ref{pro:kauffman-bound},
we conclude:
\begin{cor}
Let $p\geq3$ be odd and $q>p$. Then the characteristic algebra $\mathcal{C}(T_{p,-q})$
admits an $n$-dimensional representation for $n=2$ but not for $n=1$.\end{cor}
\begin{rem}
The knots $T_{3,-4}$ and $T_{3,-5}$ are the unique $tb$-maximizing
representatives of $m(8_{19})$ and $m(10_{124})$ up to change of
orientation \cite{Torus knots}, so if any $tb$-maximizing Legendrian
representative of a knot with at most 10 crossings has vanishing contact
homology or characteristic algebra (such as the $m(10_{132})$ of
Section \ref{sub:10_132-1}) then it must represent one of $m(9_{42})$,
$m(10_{128})$, $m(10_{132})$, or $m(10_{136})$. The characteristic
algebra of the known $tb$-maximizing Legendrian $m(9_{42})$, which
has a plat diagram with braid word \[
2,1,1,4,5,3,5,3,2,4,3,3,2,4,\]
can also be shown to have a $2$-dimensional representation, so it
does not vanish.
\end{rem}
It is not known whether there are Legendrian knots whose characteristic
algebras have representations of minimal dimension $n\geq3$, or whether
this minimal dimension can be used to distingush any Legendrian knots
with nontrivial characteristic algebras and the same classical invariants.
We leave open the question of which Legendrian knots $K$ admit representations
$\mathcal{C}(K)\to\mathrm{Mat}_{n}(\mathbb{F})$ for fixed $n\geq2$
or even for any finite $n$.

\appendix

\section{\label{sec:differential-10_132-1}The differential of the vanishing
$m(10_{132})$}

Let $K_{1}$ be the representative of $m(10_{132})$ with braid word
\[
6,7,4,3,7,5,3,6,4,2,5,1,3,2,5,2,4,6,2.\]
Then $Ch(K_{1})$ has generators $x_{1},\dots,x_{23}$ over $\mathbb{Z}[t,t^{-1}]$
with the following nonzero differentials \cite{LegInv.nb}:\begin{eqnarray*}
\partial x_{2} & = & -x_{1}\\
\partial x_{4} & = & x_{3}\\
\partial x_{6} & = & x_{3}x_{1}\\
\partial x_{8} & = & x_{3}+x_{3}x_{2}x_{5}-x_{6}x_{5}\\
\partial x_{9} & = & x_{1}+x_{7}x_{4}x_{1}-x_{7}x_{6}\\
\partial x_{11} & = & 1+x_{2}x_{5}+x_{7}x_{4}+x_{7}x_{4}x_{2}x_{5}-x_{7}x_{8}+x_{9}x_{5}\\
\partial x_{12} & = & x_{10}\\
\partial x_{13} & = & x_{10}x_{4}x_{1}-x_{10}x_{6}\\
\partial x_{14} & = & -x_{12}x_{4}x_{1}+x_{12}x_{6}+x_{13}\\
\partial x_{17} & = & x_{10}x_{4}x_{15}+x_{10}x_{4}x_{2}x_{5}x_{15}-x_{10}x_{8}x_{15}+x_{13}x_{5}x_{15}\\
\partial x_{18} & = & -x_{15}x_{7}\\
\partial x_{20} & = & 1-x_{4}x_{1}+x_{6}-x_{4}x_{1}x_{16}x_{19}+x_{6}x_{16}x_{19}\\
\partial x_{21} & = & 1-x_{12}x_{4}x_{15}-x_{12}x_{4}x_{2}x_{5}x_{15}+x_{12}x_{8}x_{15}-x_{14}x_{5}x_{15}+x_{17}\\
 &  & -x_{19}x_{5}x_{15}-x_{19}x_{16}x_{12}x_{4}x_{15}-x_{19}x_{16}x_{12}x_{4}x_{2}x_{5}x_{15}\\
 &  & +x_{19}x_{16}x_{12}x_{8}x_{15}-x_{19}x_{16}x_{14}x_{5}x_{15}+x_{19}x_{16}x_{17}\\
\partial x_{22} & = & 1-x_{10}+x_{17}x_{7}+x_{10}x_{4}x_{18}+x_{10}x_{4}x_{2}x_{5}x_{18}-x_{10}x_{8}x_{18}+x_{13}x_{5}x_{18}\\
\partial x_{23} & = & t^{-1}+x_{15}x_{2}+x_{15}x_{7}x_{4}x_{2}+x_{15}x_{9}-x_{18}x_{3}x_{2}+x_{18}x_{6}.\end{eqnarray*}

\section{\label{sec:differential-10_132-2}The differential of the nonvanishing
$m(10_{132})$}

Let $K_{2}$ be the representative of $m(10_{132})$ with braid word
\[
4,5,3,5,3,2,4,1,3,2,4,2,5,1,3,2,4,4,3,5,4,2.\]
Then $Ch(K_{2})$ has generators $x_{1},\dots,x_{25}$ over $\mathbb{Z}/2\mathbb{Z}$
with the following nonzero differentials \cite{LegInv.nb}:

\begin{eqnarray*}
\partial x_{2}=\partial x_{3} & = & x_{1}\\
\partial x_{7} & = & x_{4}+x_{5}(1+(x_{2}+x_{3})x_{4})\\
\partial x_{8} & = & x_{6}\\
\partial x_{9} & = & x_{6}(1+(x_{2}+x_{3})x_{4})\\
\partial x_{10} & = & x_{9}+x_{8}(1+(x_{2}+x_{3})x_{4})\\
\partial x_{13} & = & x_{6}(x_{2}+x_{3})+x_{11}(1+x_{5}(x_{2}+x_{3}))\\
\partial x_{14} & = & (1+(x_{2}+x_{3})x_{4})x_{12}\\
\partial x_{15} & = & x_{12}x_{11}\\
\partial x_{16} & = & x_{14}x_{11}+(1+(x_{2}+x_{3})x_{4})x_{15}\\
\partial x_{17} & = & x_{12}(x_{13}+x_{8}(x_{2}+x_{3}))+x_{15}(1+x_{5}(x_{2}+x_{3}))\\
\partial x_{19} & = & (1+(x_{2}+x_{3})x_{4})+cx_{18}\\
\partial x_{20} & = & x_{18}x_{12}\\
\partial x_{21} & = & x_{14}+x_{19}x_{12}+cx_{20}\\
\partial x_{22} & = & (1+x_{5}(x_{2}+x_{3}))x_{18}\\
\partial x_{23} & = & 1+x_{11}x_{22}+(x_{13}+x_{8}(x_{2}+x_{3}))x_{18}\\
\partial x_{24} & = & 1+x_{22}x_{12}+(1+x_{5}(x_{2}+x_{3}))x_{20}\\
\partial x_{25} & = & 1+c\end{eqnarray*}
where \[
c=x_{2}+x_{3}+(1+(x_{2}+x_{3})x_{4})x_{17}+x_{14}(x_{13}+x_{8}(x_{2}+x_{3}))+x_{16}(1+x_{5}(x_{2}+x_{3})).\]


\begin{thebibliography}{17}
\bibitem{KnotInfo}J. C. Cha and C. Livingston, KnotInfo: Table of
Knot Invariants, \url{http://www.indiana.edu/~knotinfo}, December
3, 2010.

\bibitem{Chekanov}Yuri Chekanov, \emph{Differential algebra of Legendrian
links}, Invent. Math. \textbf{150} (2002), no. 3, 441--483. 

\bibitem{atlas}Wutichai Chongchitmate and Lenhard Ng, \emph{An atlas
of Legendrian knots}, arXiv:1010.3997.

\bibitem{Eliashberg}Yakov Eliashberg, \emph{Invariants in contact
topology,} Proceedings of the International Congress of Mathematicians,
Vol. II (Berlin, 1998), Doc. Math. \textbf{1998}, Extra Vol. II, 327--338
(electronic).

\bibitem{Epstein-Fuchs}Judith Epstein and Dmitry Fuchs, \emph{On
the invariants of Legendrian mirror torus links}, Symplectic and contact
topology: interactions and perspectives (Toronto, ON/Montreal, QC,
2001), 103\textendash{}115, Fields Inst. Commun., 35, Amer. Math.
Soc., Providence, RI, 2003. 

\bibitem{Etnyre-intro}John Etnyre, \emph{Legendrian and transversal
knots}, Handbook of knot theory, 105--185, Elsevier B. V., Amsterdam,
2005. 

\bibitem{Torus knots}John Etnyre and Ko Honda, \emph{Knots and contact
geometry. I. Torus knots and the figure eight knot}, J. Symplectic
Geom. \textbf{1} (2001), no. 1, 63--120. 

\bibitem{Fuchs}Dmitry Fuchs, \emph{Chekanov-Eliashberg invariant
of Legendrian knots: existence of augmentations}, J. Geom. Phys. \textbf{47}
(2003), no. 1, 43--65.

\bibitem{Fuchs-Ishkhanov}Dmitry Fuchs and Tigran Ishkhanov, \emph{Invariants
of Legendrian knots and decompositions of front diagrams}, Mosc. Math.
J. \textbf{4} (2004), no. 3, 707\textendash{}717. 

\bibitem{FT-Kauffman}Dmitry Fuchs and Serge Tabachnikov, \emph{Invariants
of Legendrian and transverse knots in the standard contact space},
Topology \textbf{36} (1997), no. 5, 1025--1053. 

\bibitem{LegInv.nb}P. Melvin et al., Legendrian Invariants.nb, Mathematica
program available at \url{http://www.haverford.edu/math/jsabloff/Josh_Sabloff/Research.html}.

\bibitem{Ng - computable}Lenhard Ng, \emph{Computable Legendrian
invariants}, Topology \textbf{42} (2003), no. 1, 55--82.

\bibitem{Ng - 2-bridge}Lenhard Ng, \emph{Maximal Thurston-Bennequin
number of two-bridge links}, Algebr. Geom. Topol. \textbf{1} (2001),
427--434 (electronic).

\bibitem{Ng - max tb}Lenhard Ng, \emph{On arc index and maximal Thurston-Bennequin
number}, arXiv:math.GT/0612356.

\bibitem{Rutherford}Dan Rutherford, \emph{Thurston-Bennequin number,
Kauffman polynomial, and ruling invariants of a Legendrian link: the
Fuchs conjecture and beyond}, Int. Math. Res. Not. \textbf{2006},
Art. ID 78591, 15 pp.

\bibitem{Sabloff}Joshua Sabloff, \emph{Augmentations and rulings
of Legendrian knots}, Int. Math. Res. Not. \textbf{2005}, no. 19,
1157--1180. 

\bibitem{SV}Clayton Shonkwiler and David Shea Vela-Vick, \emph{Legendrian
contact homology and nondestabilizability}, J. Symplectic Geom. \textbf{9}
(2011), no. 1, 1--12.
\end{thebibliography}
\end{document}